\documentclass[11pt]{article}

\usepackage{enumerate}
\usepackage{amsmath}
\usepackage{amssymb,latexsym}
\usepackage{amsthm}
\usepackage{color}
\usepackage{cancel}
\usepackage{graphicx}

\usepackage{hyperref}

\usepackage{amscd}

\DeclareMathOperator{\C}{\mathcal{C}}

\newtheorem{theorem}{Theorem}[section]

\newtheorem{corollary}[theorem]{Corollary}

\newtheorem{proposition}[theorem]{Proposition}

\newtheorem{remark}[theorem]{Remark}

\newcommand{\fqn}{\mathbb{F}_{q^n}}

\newcommand{\cC}{{\mathcal C}}

\newcommand{\cL}{{\mathcal L}}

\newcommand{\F}{{\mathbb F}}

\newcommand{\Tr}{{\mathrm Tr}}

\newcommand{\K}{{\mathbb K}}
\newcommand{\fq}{{\mathbb F}_{q}}

\newcommand{\la}{\langle}
\newcommand{\ra}{\rangle}

\newcommand{\PG}{\mathrm{PG}}
\newcommand{\N}{\mathrm{N}}

\title{On the intersection problem for linear sets in the projective line}

\author{Giovanni Zini and Ferdinando Zullo\thanks{
The second author was supported by the project ''VALERE: VAnviteLli pEr la RicErca" of the University of Campania ''Luigi Vanvitelli''.
The research was also supported by the Italian National Group for Algebraic and Geometric Structures and their Applications (GNSAGA
- INdAM). }}
\date{}

\begin{document}
\maketitle

\begin{abstract}
The aim of this paper is to investigate the intersection problem between two linear sets in the projective line over a finite field.
In particular, we analyze the intersection between two clubs with eventually different maximum fields of linearity.
Also, we analyze the intersection between the linear set defined by the polynomial $\alpha x^{q^k}+\beta x$ and other linear sets having the same rank; this family contains the linear set of pseudoregulus type defined by $x^q$.
The strategy relies on the study of certain algebraic curves whose rational points describe the intersection of the two linear sets.
Among other geometric and algebraic tools, function field theory and the Hasse-Weil bound play a crucial role.
As an application, we give asymptotic results on semifields of BEL-rank two.
\end{abstract}

\bigskip
{\it AMS subject classification:} 51E20, 05B25, 51E22

\bigskip
{\it Keywords:} Linear set, linearized polynomial, algebraic curve

\section{Introduction}

Let $q$ be a prime power and $r,n$ be two positive integers.
Consider two $\F_q$-linear sets $L_1$ and $L_2$ in the projective space $\PG(r,q^n)$. Clearly, $L_1 \cap L_2$ is still an $\F_q$-linear set of $\PG(r,q^n)$, whenever $L_1 \cap L_2$ is non-empty.
Hence, the {\bf intersection problem} can be formulated as follows:
\begin{enumerate}
    \item do $L_1$ and $L_2$ meet in at least one point?
    \item if $L_1\cap L_2\ne\emptyset$, what is the size of $L_1\cap L_2$?
\end{enumerate}
An answer to these questions turn out to be difficult in general.
The intersection problem for linear sets has been investigated in the following cases:
\begin{itemize}
    \item two subspaces;
    \item two subgeometries, see \cite{DD2008};
    \item one $\F_q$-linear set with one $\F_q$-subline in $\PG(1,q^n)$, see \cite{LVdV2010,Pepe};
    \item two scattered linear sets of rank $n+1$ in $\PG(2,q^n)$, \cite{DD2014};
    \item one scattered linear set of rank $3n$ with either one line or one plane in $\PG(2n-1,q^3)$, see \cite{LVdV2013};
    \item two clubs, see \cite{ShVVdV}.
\end{itemize}

The intersection problem for linear sets appears in several contexts, such as blocking sets, KM-arcs, PN-functions, irreducible polynomials and semifields, see e.g. \cite{Card,DeBoeckVdV2016,Lavrauw,Polverino,ShVVdV}.

This paper is devoted to the problem of intersecting two linear sets of rank $n$ in the projective line $\PG(1,q^n)$.
The aim is to fix one linear set $L_1$ and then to provide sufficient conditions on $L_2$ ensuring at least one common point.
More precisely, fix a $q$-polynomial $g(x)$ over $\F_{q^n}$ and let $f(y)$ be another $q$-polynomial over $\F_{q^n}$ and  $h$ be a non negative integer.
Consider the following $\F_q$-linear sets:
\[ L_g=\{ \langle (x,g(x)) \rangle_{\F_{q^n}} \colon x \in \F_{q^{n}}^* \}, \quad  L_f=\{ \langle (y^{q^{h}},f(y)) \rangle_{\F_{q^n}} \colon y \in \F_{q^{n}}^* \}, \]
\[ L_f'=\{ \langle (f(y),y^{q^{h}}) \rangle_{\F_{q^n}} \colon y \in \F_{q^{n}}^* \}. \]
We translate the intersection problem into the study of certain algebraic curves, as follows.
\begin{itemize}
    \item $L_g\cap L_f \ne \emptyset$ if and only if the curve with plane model
    \[\frac{g(X)}X-\frac{f(Y)}{Y^{q^h}}=0 \]
    has at least one $\F_{q^n}$-rational affine point with nonzero coordinates.
    \item $L_g\cap L_f' \ne \emptyset$ if and only if the curve with plane model
    \[\frac{g(X)}X-\frac{Y^{q^h}}{f(Y)}=0 \]
    has at least one $\F_{q^n}$-rational affine point with nonzero coordinates.
\end{itemize}
We then use algebraic and geometric machinery, such as polynomial manipulations, extensions of function fields (see Theorems \ref{th:kummer} and \ref{th:artinschreier}) and the Hasse-Weil bound (see Theorem \ref{th:hasseweil}), to obtain conditions involving $\deg f$, $\deg g$, $h$ and $n$, which ensure nonempty intersection between $L_g$ and either $L_f$ or $L_f'$.

The use of algebraic curves in the study of linear sets has already proved in the literature to be fruitful to get classification and asymptotic results; see e.g. \cite{BM,BZ,BZ2}.

The paper is organized as follows.
The original results are resumed in Section \ref{sec:results}.
Section \ref{sec:prelim} recalls  preliminary results and tools about linear sets (Section \ref{sec:linearset}) and algebraic curves (Section \ref{sec:algebraiccurve}).
Section \ref{sec:intersections} contains the main results of the paper. In particular, Section \ref{sec:clubs} deals with the intersection of two clubs, namely  $h=0$, $g(x)=\mathrm{Tr}_{q^n/q^{r_1}}(x)$, and $f(y)=\alpha \mathrm{Tr}_{q^n/q^{r_2}}(y)$. Section \ref{sec:binomial} investigates the case $g(x)=\alpha x^{q^k}+\beta x$; note that, when $g(x)=x^q$, $L_g$ is a linear set of so-called pseudoregulus type.
Finally, Section \ref{sec:semifields} applies results of Section \ref{sec:intersections} to give asymptotic results about semifields with BEL-rank two.

\subsection{Original results of the paper}\label{sec:results}
Firstly, we study the intersection between two distinct clubs in $\PG(1,q^n)$, that is, we choose $h=0$, $g(x)=\mathrm{Tr}_{q^n/q^{r_1}}(x)$ and $f(y)=\alpha \mathrm{Tr}_{q^n/q^{r_2}}(y)$, with $r_1,r_2$ two positive integers such that $r_1,r_2 \mid n$ and $\alpha \in \F_{q^n}^*$.
Without loss of generality, we may assume $(r_1,r_2)=1$ (see Remark \ref{rk:d=1}).
With this notation, we prove the following results.
\begin{itemize}
    \item (Theorem \ref{th:primo}) If there exists $a\in \F_{q^n}$ such that $\mathrm{Tr}_{q^n/q^{r_1}}(a)=-1$ and $\mathrm{Tr}_{q^n/q^{r_2}}(\alpha a)=1$, or such that $\mathrm{Tr}_{q^n/q^{r_2}}(a)=-1$ and $\mathrm{Tr}_{q^n/q^{r_1}}(a/\alpha)=1$,
then $|L_g\cap L_f|\geq 2$.
\item (Theorem \ref{th:secondo}) If $\alpha=a\cdot b$, with $a \in \F_{q^{r_1}}$ and $b \in \F_{q^{r_2}}$, then $|L_{g} \cap L_{f}|\geq 2$  if and only if there exist
$\gamma_1,\gamma_2\in\F_{q^n}$ such that $\mathrm{Tr}_{q^n/q^{r_1}}(\gamma_1)=\mathrm{Tr}_{q^n/q^{r_2}}(\gamma_2)=1$ and $\mathrm{Tr}_{q^n/q}(a\gamma_1-\frac{\gamma_2}{b})=0$.
\item (Theorem \ref{th:Cr1r2}) If $r_1+r_2+1 \leq \frac{n}2$, then $|L_g\cap L_f'|\geq 1$.
\end{itemize}

Secondly, we study the case $g(x)=\alpha x^{q^k}+\beta x$, with $\alpha,\beta\in\F_{q^n}$, $\alpha\ne0$, $k\geq1$.
Write $f(y)=\sum_{i=0}^d a_i y^{q^i}$, with $d<n$, $a_d\ne0$. Define $\ell=\min\{i\colon a_i\ne0\}$; when $f(y)$ is not a monomial, define \[\ell_2=\min\{i>\ell\colon a_i\ne0\},\quad \ell_3=\max\{i<d\colon a_i\ne0\}.\]
If $h=d$, define $t=\ell$; if $h=\ell$, define $t=d$.
We prove the following necessary and sufficient conditions.
\begin{itemize}
    \item (Proposition \ref{prop:h=dmon} and Theorem \ref{th:mon}) Let $f(y)$ be a monomial and suppose $h=d$ or $\beta=0$. Then $L_g\cap L_f\ne\emptyset$ if and only if $\mathrm{N}_{q^n/q^e}\left(\frac{a_d-\beta}{\alpha}\right)=1$, where $e=\gcd(n,k,d-h)$.
    \item (Proposition \ref{prop:sabbassa}) Let $f(y)$ be a binomial and suppose that either $d=h$ and $a_d=\beta$, or $\ell=h$ and $a_\ell=\beta$. Then $L_g\cap L_f\ne\emptyset$ if and only if $\mathrm{N}_{q^n/q^e}\left(\frac{a_t}{\alpha}\right)=1$, where $e=\gcd(n,k,t-h)$.
    \item (Proposition \ref{prop:pointinv}) Let $f(y)$ be a monomial and suppose $h=d$ or $\beta=0$. Then $L_g\cap L_f^\prime\ne\emptyset$ if and only if $\mathrm{N}_{q^n/q^e}\left(\frac{1-\beta a_d}{\alpha a_d}\right)=1$, where $e=\gcd(n,k,d-h)$.
\end{itemize}
For the remaining cases of $f(y)$, we prove that $L_g\cap L_f\ne\emptyset$ if one of the following conditions holds.
\begin{itemize}
    \item (Theorem \ref{th:mon}) $f(y)$ is a monomial, $h\ne d$, $\beta\ne0$ and $k+|d-h|\leq n/2$.
    \item (Theorem \ref{th:h<ell_dritto}) $f(y)$ is not a monomial; $h\leq\ell$; when $f(y)$ is a binomial and $h=\ell$, assume $a_h\ne\beta$; and
    \[
\max\left\{ k+d-h-m_h,\frac{d-h}{2}\right\}\leq\begin{cases} \frac{n}{2} & \textrm{if}\quad m_h\leq\frac{d-h}{2}, \\ \frac{n}{2}-1 & \textrm{if}\quad m_h>\frac{d-h}{2}, \end{cases}
\]
 where
    \[
m_h=\begin{cases}
0,&\textrm{if}\quad a_h\ne\beta, \\
\ell-h,&\textrm{if}\quad a_h=\beta=0, \\
\ell_2-h,&\textrm{if}\quad a_h=\beta\ne0. \\
\end{cases}
\]
    \item (Theorem \ref{th:h>ell}) $f(y)$ is not a monomial; $h>\ell$; when $f(y)$ is a binomial and $h=\ell$, assume $a_h\ne\beta$; and $k+m_h-\ell\leq n/2$, where
    \[
 m_h=\begin{cases} \max\{d,h\} & \textrm{if}\quad a_d\ne\beta\;\textrm{ or }\;d\ne h, \\ \ell_3 & \textrm{if}\quad a_d=\beta\;\textrm{ and }\; d=h. \end{cases}
\]
\end{itemize}
Also, $L_g\cap L_f^\prime\ne\emptyset$ if one of the following cases holds.
\begin{itemize}
    \item (Proposition \ref{prop:pointinv}) $f(y)$ is a monomial, $h\ne d$, $\beta\ne0$ and $k+|d-h|\leq n/2$.
    \item (Theorem \ref{th:h<ell_girato}) $f(y)$ is not a monomial, $h\leq\ell$ and $k+d-\min\{m_h,\ell\}+1\leq n/2$, where
    \[
m_h=\begin{cases}
\ell_2 & \textrm{if}\quad \beta a_h=1, \\ h & \textrm{if}\quad \beta\ne0\quad \textrm{and}\quad\beta a_h\ne1, \\
d & \textrm{if}\quad \beta=0.
\end{cases}
\]
    \item (Theorem \ref{th:h>lrib}) $f(y)$ is not a monomial, $h>\ell$ and $k+\max\{m_h,d\}-\ell+1\leq n/2$, where
\[
m_h=\begin{cases}
\ell, & \textrm{if}\quad \beta=0;\\
\max\{d,h\}, & \textrm{if}\quad \beta\ne0\quad \textrm{and either}\quad h\ne d\quad\textrm{or}\quad \beta a_h\ne1;\\
\ell_3,&\textrm{if}\quad h=d\quad\textrm{and}\quad \beta a_h=1.\\
\end{cases}
\]
\end{itemize}

As pointed out in Remark \ref{rem:aggiunto}, we make use of the adjoint operation on $f(y)$ to prove Corollaries \ref{cor:h<ell_dritto_aggiunto}, \ref{cor:h>elldual}, \ref{cor:rovultimo}, \ref{cor:ultimissimo}, where sufficient conditions for $L_g\cap L_f\ne\emptyset$ and $L_g\cap L_f^\prime\ne\emptyset$ are obtained as the application of the results listed above.

As a consequence of the above results, we can describe the intersection between the linear set $L_{x^q}=\{\langle (x,x^q) \rangle_{\F_{q^n}} \colon x \in \F_{q^n}^*\}$ of pseudoregulus type and the linear sets $L_f$ and $L_f^\prime$.

\smallskip

If $f(y)$ is a monomial, then $L_{x^q}\cap L_{f}\ne\emptyset$ is equivalent to $\mathrm{N}_{q^n/q}(a_d)=1$, and this is also equivalent to $L_{x^q}\cap L_f^\prime\ne\emptyset$.

\smallskip

We have $L_{x^q}\cap L_f\ne\emptyset$ whenever $f(y)$ is not a monomial and one of the following conditions is satisfied.

\noindent{\bf Condition (I):}  $h\leq\ell$ and one of the following cases holds.

\begin{itemize}
    \item[(I.1)] $\max\left\{ d-h+1-m_h,\frac{d-h}{2}\right\}\leq\begin{cases} \frac{n}{2} & \textrm{if}\quad m_h\leq\frac{d-h}{2}, \\ \frac{n}{2}-1 & \textrm{if}\quad m_h>\frac{d-h}{2}, \end{cases}$ \\
        where $m_h=\begin{cases}
0,&\textrm{if}\quad a_h\ne0, \\
\ell-h,&\textrm{if}\quad a_h=0. \\
\end{cases}
$
    \item[(I.2)] $h=\ell=0$ and $\ell_2-1\geq n/2$.
\end{itemize}
\noindent{\bf Condition (II):} $h>\ell$ and one of the following cases holds.
\begin{itemize}
    \item[(II.1)] $\max\{d,h\}-\ell+1\leq n/2$.
    \item[(II.2)] $\ell\ne0$, $h\geq d$, and $\max\left\{ h-\ell+1-\hat{m}_h,\frac{h-\ell}{2}\right\}\leq\begin{cases} \frac{n}{2} & \textrm{if}\quad \hat{m}_h\leq\frac{h-\ell}{2}, \\ \frac{n}{2}-1 & \textrm{if}\quad \hat{m}_h>\frac{h-\ell}{2}, \end{cases}$ \\
        where $\hat{m}_h=\begin{cases}
0,&\textrm{if}\quad a_h\ne0, \\
h-d,&\textrm{if}\quad a_h=0. \\
\end{cases}$
    \item[(II.3)] $\ell=0$ and   $\min\{h,\ell_2\}-1\geq n/2$.
\end{itemize}

We have $L_{x^q}\cap L_f^\prime\ne\emptyset$ whenever $f(y)$ is not a monomial and one of the following conditions is satisfied: either $d-\ell+2\leq n/2$, or $\ell=0$ and $n-\ell_2+2\leq n/2$.


\section{Preliminaries results and tools}\label{sec:prelim}

Through this paper, $q$ is a power of a prime $p$, $\F_q$ denotes the finite field of order $q$, $\F_{q^n}$ denotes its finite extension of degree $n\geq 2$ and $\K$ is its algebraic closure.
Also, if $a,b,c$ are integers, we denote by $(a,b)$ (resp. $(a,b,c)$) the greatest common divisor of $a$ and $b$ (resp. of $a$, $b$, and $c$), taken as a positive integer.

\subsection{Linear sets}\label{sec:linearset}

Let $\Lambda=\PG(V,\F_{q^n})=\PG(1,q^n)$, where $V$ is a vector space of dimension $2$ over $\F_{q^n}$.
A point set $L$ of $\Lambda$ is said to be an \emph{$\F_q$-linear set} of $\Lambda$ of rank
$k$ if $L$ is
defined by the non-zero vectors of a $k$-dimensional $\F_q$-vector subspace $U$ of $V$, i.e.
\[L=L_U=\{\la {\bf u} \ra_{\mathbb{F}_{q^n}} \colon {\bf u}\in U\setminus \{{\bf 0} \}\}.\]
The number of points of $L_U$ is upper bounded by $\frac{q^k-1}{q-1}$ and $L_U$ is called \emph{scattered} if it has the maximum number of points (w.r.t. its rank).
Also, recall that the \emph{weight of a point} $P=\langle \mathbf{u} \rangle_{\F_{q^n}}$ is $w_{L_U}(P)=\dim_{\F_q}(U\cap\langle \mathbf{u} \rangle_{\F_{q^n}})$. A linear set is scattered if and only if each of its points has weight one.

Let $L_U$ be an $\F_q$-inear set of rank $n$ of $\PG(1,q^n)$. Up to projectivity, we can assume that $\la (0,1) \ra_{\F_{q^n}}\notin L_U$. Then $U=U_f:=\{(x,f(x))\colon x\in \F_{q^n}\}$ for some $q$-polynomial $f(x)=\sum_{i=0}^{n-1}a_ix^{q^i}$; we consider $f(x)$ as an element of $\tilde{\cL}_{n,q}=\cL_{n,q}/(x^{q^n}-x)$, where $\cL_{n,q}\subset\F_{q^n}[x]$ is the $\F_q$-algebra of $q$-polynomials. In this case, we denote $L_U$ by $L_f$.

For any positive divisor $r$ of $n$, let $\mathrm{Tr}_{q^n/q^r}(x)=\sum_{i=0}^{n/r-1} x^{q^{ir}}$ and $\mathrm{N}_{q^n/q^r}(x)=x^{\frac{q^n-1}{q^r-1}}$.
Consider the non-degenerate symmetric bilinear form of $\F_{q^n}$ over $\F_q$ defined for every $x,y \in \F_{q^n}$ by
\begin{equation}\label{eq:bilform} \la x,y\ra= \mathrm{Tr}_{q^n/q}(xy). \end{equation}
\noindent Then the \emph{adjoint} $\hat{f}$ of the linearized polynomial $\displaystyle f(x)=\sum_{i=0}^{n-1} a_ix^{q^i} \in \tilde{\mathcal{L}}_{n,q}$ with respect to the bilinear form $\la\cdot,\cdot\ra$ is
\[ \hat{f}(x)=\sum_{i=0}^{n-1} a_i^{q^{n-i}}x^{q^{n-i}}, \]
i.e.
\[ \mathrm{Tr}_{q^n/q}(xf(y))=\mathrm{Tr}_{q^n/q}(y\hat{f}(x)), \]
for every $x,y \in \F_{q^n}$.

\begin{proposition}\label{prop:adjoint}{\rm (\cite[Lemma 2.6]{BGMP2015},\cite[Lemma 3.1]{CsMP})}
Let $L_f$ be an $\F_q$-linear set of rank $n$ in $\PG(1,q^n)$, with $f(x)\in \tilde{\mathcal{L}}_{n,q}$, and let $\hat{f}(x)$ be its adjoint w.r.t. the bilinear form \eqref{eq:bilform}. Then $L_f=L_{\hat{f}}$.
\end{proposition}

Recently, the class of scattered linear sets has been intensively studied because of their connections with MRD-codes, see \cite{CsMPZ2019,CSMPZ2016,Lunardon2017,Sheekey2016,ShVdV}.
One of the first families of scattered linear sets was found by Blokhuis and Lavrauw in  \cite{BL2000}. They are known as linear sets of \emph{pseudoregulus type} and
 can be defined as any linear set which is
 $\mathrm{P\Gamma L}(2,q^n)$-equivalent to
\[ L_{p}=\{\la (x,x^q) \ra_{\mathbb{F}_{q^n}} \colon x \in \F_{q^n}^*\}, \]
see \cite[Section 4]{LuMaPoTr2014}.
A further important class of linear sets is given by the clubs.
A \emph{club} $L_U$ in $\PG(1,q^n)$ is an $\F_q$-linear set of rank $n$ such that all but one point of $L_U$ have weight one, and one point (called the \emph{head} of $L_U$) has weight $n-1$, see \cite{Szi}.
Clubs appear in the construction of KM-arcs, see \cite{DeBoeckVdV2016}.

For further details on linear sets we refer to \cite{Lavrauw,LVdV2015,Polverino}.

\subsection{Algebraic curves and function fields}\label{sec:algebraiccurve}

Let $\mathbb{L}$ be either $\mathbb{F}_{q^n}$ or $\mathbb{K}=\overline{\F}_{q^n}$; hence, $\mathbb{L}$ is a perfect field.
An algebraic function field in one variable (briefly a function field) over $\mathbb{L}$ is an extension field $\mathbb{F}$ of $\mathbb{L}$ with transcendence degree $1$.
Details on the theory of function fields may be found in \cite{Sti}, to which we refer for basic definitions and results.
We assume that $\mathbb{L}$ is the full constant field of $\mathbb{F}$, i.e. $\mathbb{L}$ is algebraically closed in $\mathbb{F}$.
We denote by $\mathbb{P}(\mathbb{F})$ the set of places of $\mathbb{F}$, and by $v_P(z)\in\mathbb{Z}$ the valuation of $z\in\mathbb{F}$ at $P\in\mathbb{P}(\mathbb{F})$.
We denote by $\deg(P)$ the degree of $P\in\mathbb{P}(\mathbb{F})$, i.e. the degree $\left[\mathbb{F}_P\colon\mathbb{L}\right]$, where $\mathbb{F}_P$ is the residue class field of $P$; if $\mathbb{L}=\mathbb{K}$, then $\deg(P)=1$.
For any finite extension $\mathbb{F}^\prime$ of $\mathbb{F}$, we write $P^\prime\mid P$ when the place $P^\prime\in\mathbb{P}(\mathbb{F}^\prime)$ lies over the place $P\in\mathbb{P}(\mathbb{F})$, and we denote by $e(P^\prime\mid P)\geq1$ the ramification index of $P^\prime$ over $P$; if $e(P^\prime\mid P)=1$ for every $P^\prime$ lying over $P$, then $P$ is said to be unramified in $\mathbb{F}^\prime\colon\mathbb{F}$.

We recall two important types of extensions of function fields over $\mathbb{L}$.

\begin{theorem}\label{th:kummer}{\rm \cite[Corollary 3.7.4]{Sti}}
Let $m$ be a positive integer such that $p\nmid m$ and $\mathbb{L}$ contains a primitive $m$-th root of unity.
Let $u\in\mathbb{F}$ be such that there exists $Q\in\mathbb{P}(\mathbb{F})$ satisfying $(m,v_Q(u))=1$.
Let $\mathbb{F}^\prime=\mathbb{F}(x)$, where $x$ is a root of $\Phi(T)=T^m-u\in\mathbb{F}[T]$.
Then
\begin{itemize}
    \item $\Phi(T)$ is irreducible over $\mathbb{F}$, and the field extension $\mathbb{F}^\prime\colon\mathbb{F}$ of degree $m$ is called a \emph{Kummer extension};
    \item for any $P\in\mathbb{P}(\mathbb{F})$ and $P^\prime\in\mathbb{P}(\mathbb{F}^\prime)$ with $P^\prime\mid P$, we have $e(P^\prime\mid P)=m/r_P$, where $r_P=(m,v_P(u))$;
    \item $\mathbb{L}$ is the full constant field of $\mathbb{F}^\prime$;
    \item the genera $g$ and $g^\prime$ of $\mathbb{F}$ and $\mathbb{F}^\prime$ (respectively) satisfy
    \[g^\prime=1+m(g-1)+\frac{1}{2}\sum_{P\in\mathbb{P}(\mathbb{F})}(m-r_P)\deg(P).\]
\end{itemize}
\end{theorem}

\begin{theorem}\label{th:artinschreier}{\rm \cite[Theorem 3.7.10]{Sti}}
Let $L(T)\in\mathbb{L}[T]$ be a separable $p$-polynomial of degree $\bar{q}$ with all its roots in $\mathbb{L}$.
Let $u\in\mathbb{F}$ be such that for every $P\in\mathbb{P}(\mathbb{F})$ there exists $z\in\mathbb{F}$ (depending on $P$) satisfying either $v_P(u-L(z))\geq0$ or $v_P(u-L(z))=-m$ with $m>0$ and $p\nmid m$.
Define $m_P=-1$ in the former case and $m_P=m$ in the latter case.
Let $\mathbb{F}^\prime=\mathbb{F}(x)$, where $x$ is a root of $\Phi(T)=L(T)-u\in\mathbb{F}[T]$.
If there exists $Q\in\mathbb{P}(\mathbb{F})$ such that $m_Q>0$, then
\begin{itemize}
    \item $\Phi(T)$ is irreducible over $\mathbb{F}$, and the field extension $\mathbb{F}^\prime\colon\mathbb{F}$ of degree $\bar{q}$ is called a \emph{generalized Artin-Schreier extension};
    \item each place $P\in\mathbb{P}(\mathbb{F})$ is either unramified or totally ramified in $\mathbb{F}^\prime\colon\mathbb{F}$ according to $m_P=-1$ or $m_P>0$, respectively;
    \item $\mathbb{L}$ is the full constant field of $\mathbb{F}^\prime$;
    \item the genera $g$ and $g^\prime$ of $\mathbb{F}$ and $\mathbb{F}^\prime$ (respectively) satisfy
    \[g^\prime= \bar{q}\cdot g+\frac{\bar{q}-1}{2}\left(-2+\sum_{P\in\mathbb{P}(\mathbb{F})}(m_P+1)\deg(P)\right).\]
\end{itemize}
\end{theorem}

Let $\cC$ be a projective algebraic plane curve over $\mathbb{K}$ with affine equation $F(X,Y)=0$. Suppose that $\cC$ is defined over $\mathbb{F}_{q^n}$, i.e. the ideal of $\cC$ is generated by a polynomial over $\mathbb{F}_{q^n}$, and hence we can assume $F(X,Y)\in\mathbb{F}_{q^n}[X,Y]$.
If $F(X,Y)$ is irreducible over $\mathbb{K}$, we say that $\cC$ is absolutely (or geometrically) irreducible; in this case, we denote respectively by $\mathbb{K}(\cC)$ and $\mathbb{F}_{q^n}(\cC)$ the $\mathbb{K}$- and $\mathbb{F}_{q^n}$-rational function field of $\cC$, i.e. the fields of rational functions on $\cC$ defined over $\mathbb{K}$ and $\mathbb{F}_{q^n}$.
The fields $\mathbb{K}(\cC)$ and $\mathbb{F}_{q^n}(\cC)$ are algebraic function fields whose full constant field $\mathbb{L}$ is equal to $\K$ and $\F_{q^n}$, respectively; they are generated over $\mathbb{L}$ by the coordinate functions $x$ and $y$, which are transcendental over $\mathbb{L}$ and are algebraically dependent by $F(x,y)=0$.
Note that $F(x,T)$ and $F(T,y)$ are the minimal polynomials of $y$ and $x$ over $\mathbb{L}(x)$ and $\mathbb{L}(y)$, respectively.

The genus of the absolutely irreducible curve $\cC$ coincides with the genus of the function field $\K(\cC)$, and with the genus of the function field $\F_{q^n}(\cC)$.

Let $\mathbb{P}(\cC)=\cC(\mathbb{K})$ be the set of places of $\cC$, i.e. the set of (rational) places of $\K(\cC)$; let $\cC(\F_{q^n})$ be the set of $\mathbb{F}_{q^n}$-rational places of $\cC$, i.e. the set of rational places of $\F_{q^n}(\cC)$.
The extension $\mathbb{K}(\cC)\colon\mathbb{F}_{q^n}(\cC)$ is a constant field extension, and the $\F_{q^n}$-rational places of $\cC$ are the restriction to $\F_{q^n}(\cC)$ of places of $\K(\cC)$ which are fixed by the Frobenius map on $\mathbb{K}(\cC)$.
The center of an $\F_{q^n}$-rational place is an $\F_{q^n}$-rational point of $\cC$, that is, a point with coordinates in $\F_{q^n}$.
Conversely, if $P$ is a simple $\F_{q^n}$-rational point of $\cC$, then $P$ is the center of just one place $\mathcal{P}$ of $\cC$, and $\mathcal{P}$ is $\F_{q^n}$-rational (hence, we may identify $P$ and $\mathcal{P}$).

We now recall the well-known Hasse-Weil bound.

\begin{theorem}\label{th:hasseweil}{\rm \cite[Theorem 5.2.3]{Sti} (Hasse-Weil bound)}
Let $\mathcal{C}$ be an absolutely irreducible curve of genus $g$ defined over the finite field $\mathbb{F}_{q^n}$.
Then
\[q^n+1-2g\sqrt{q^n}\leq|\mathcal{C}(\mathbb{F}_{q^n})|\leq q^n+1+2g\sqrt{q^n}.\]
\end{theorem}

\section{Intersections of linear sets on the projective line}\label{sec:intersections}

Let $L_f$ and $L_g$ two $\F_q$-linear set of rank $n$ in $\PG(1,q^n)$, for some $q$-polynomials $f(x)=\sum_{i=0}^d a_i x^{q^i}$ and $g(x)=\sum_{i=0}^{k} b_i x^{q^i}$ in $\tilde{\mathcal{L}}_{n,q}$.
Then $L_f$ and $L_g$ intersect if and only if
\[ \bar{x}f(\bar{y})-\bar{y}g(\bar{x})=0 \]
for some $\bar{x},\bar{y}\in \F_{q^n}^*$; equivalently, if and only if the $\fqn$-rational plane curve $\mathcal{C}_{g,f}$ with affine equation
\begin{equation}\label{eq:curve}
\mathcal{C}_{g,f} \colon\quad  \frac{g(X)}X-\frac{f(Y)}Y=0
\end{equation}
has at least one $\F_{q^n}$-rational affine point with nonzero coordinates.

One of the tools that will be used is the Hasse-Weil lower bound, which can be successfully applied when the underlying curve has low degree.
Since the degree of $\C_{g,f}$ heavily depends on the degrees of $g(x)$ and $f(x)$, it is sometimes convenient to write one of the two linear sets, say $L_f$, as
\[ L_f=\{ \langle (x^{q^h},f'(x)) \rangle_{\F_{q^n}} \colon x \in \F_{q^n}^* \}, \]
for some positive integer $h$.
In this way, $L_f$ and $L_g$ meet if and only if the $\fqn$-rational curve $\mathcal{C}_{g,f}^h$ with plane model
\begin{equation}\label{eq:curve2}
\mathcal{C}_{g,f}^h \colon\quad  \frac{g(X)}X-\frac{f'(Y)}{Y^{q^h}}=0
\end{equation}
has at least one $\F_{q^n}$-rational affine point with nonzero coordinates.

We are also interested in the intersection of the two $\fq$-linear sets $L_g$ and $\sigma(L_f)=\{ \langle (f'(x),x^{q^h}) \rangle_{\F_{q^n}} \colon x \in \F_{q^n}^*\}$, where
\[\sigma\in{\rm PGL}(2,q^n),\quad \sigma:\langle(X_0,X_1)\rangle_{\F_{q^n}}\mapsto\langle(X_1,X_0)\rangle_{\F_{q^n}}.\]
The intersection of $L_g$ and $\sigma(L_f)$ is given by the $\F_{q^n}$-rational affine points with nonzero coordinates of the curve with plane model
\[
\mathcal{X}_{g,f}^h \colon\quad  \frac{g(X)}X-\frac{Y^{q^h}}{f'(Y)}=0.
\]

\subsection{Intersection of clubs}\label{sec:clubs}

Every $\F_q$-linear club of rank $n$ in $\PG(1,q^n)$ with head $\langle (1,0) \rangle_{\F_{q^n}}$ and not passing through $\langle (0,1) \rangle_{\F_{q^n}}$ has shape
\begin{equation}\label{eq:club}
\{ \langle (x, \alpha \mathrm{Tr}_{q^n/q}(x))\rangle_{\F_{q^n}} \colon x \in \F_{q^n}^* \},
\end{equation}
for some $\alpha \in \F_{q^n}^*$, see \cite[Proposition 6.1]{ShVVdV}.

We investigate the intersection between two clubs of shape \eqref{eq:club} with possibly different maximum fields of linearity $\F_{q^{r_1}},\F_{q^{r_2}}\subseteq \F_{q^n}$.
Consider the following linear sets in $\PG(1,q^n)$:
\[\{\langle ( x, \alpha_1\mathrm{Tr}_{q^n/q^{r_1}}(x) ) \rangle_{\F_{q^n}} \colon x \in \F_{q^n}^* \}, \quad \{\langle ( x, \alpha_2\mathrm{Tr}_{q^n/q^{r_2}}(x) ) \rangle_{\F_{q^n}} \colon x \in \F_{q^n}^* \}, \]
with $\alpha_1,\alpha_2\in\mathbb{F}_{q^n}^*$.
The projectivity $\langle(X_0,X_1)\rangle_{\F_{q^n}}\mapsto \langle(X_0,\alpha_1^{-1}X_1)\rangle_{\F_{q^n}}$ maps the aforementioned linear sets to
\[L_{r_1}=\{\langle ( x,\mathrm{Tr}_{q^n/q^{r_1}}(x) ) \rangle_{\F_{q^n}} \colon x \in \F_{q^n}^* \},\]
\[L_{r_2}= \{\langle ( x, \alpha\mathrm{Tr}_{q^n/q^{r_2}}(x) ) \rangle_{\F_{q^n}} \colon x \in \F_{q^n}^* \}, \]
where $\alpha=\alpha_2/\alpha_1$.
Both $L_{r_1}$ and $L_{r_2}$ are clubs with head $H=\langle(1,0)\rangle_{\F_{q^n}}$.
We investigate whether $L_{r_1}$ and $L_{r_2}$ share at least one point other than the head.

\begin{remark}\label{rk:d=1}
Let $d=(r_1,r_2)$. Then $L_{r_1},L_{r_2},\sigma(L_{r_2})$ are $\mathbb{F}_{q^d}$-linear, and hence we can replace $r_1$ with $r_1/d$, $r_2$ with $r_2/d$, and $n$ with $n/d$. Therefore, in the study of $L_{r_1}\cap L_{r_2}$ and $L_{r_1}\cap\sigma(L_{r_2})$ we will assume without restriction that $(r_1,r_2)=1$.
\end{remark}

The curve given in \eqref{eq:curve} which describes $L_{r_1}\cap L_{r_2}$ has high degree, namely $q^{n-\min\{r_1,r_2\}}-1$. The following remark allows to study a curve of lower degree, namely $q^{\max\{r_1,r_2\}}$.

\begin{remark}\label{rk:Hilbert}
Let $n$ and $r$ be positive integers with $r\mid n$.
The elements $u \in \F_{q^n}$ of shape $\frac{v}{\mathrm{Tr}_{q^n/q^r}(v)}$, where $v$ ranges over the elements of $\F_{q^n}^*$ with $\mathrm{Tr}_{q^n/q^r}(v)\ne 0$, are exactly the elements of $\F_{q^n}^*$ satisfying $\mathrm{Tr}_{q^n/q^r}(u)=1$.
By Hilbert's 90 theorem, they coincide with the elements $w^{q^r}-w+\gamma$, where $w$ ranges in $\F_{q^n}$ and $\gamma$ is a fixed element in $\F_{q^n}$ with $\mathrm{Tr}_{q^n/q^r}(\gamma)=1$.
\end{remark}

\begin{theorem}\label{th:primo}
Let $r_1,r_2\mid n$ with $(r_1,r_2)=1$. If there exists $a\in \F_{q^n}$ such that $\mathrm{Tr}_{q^n/q^{r_1}}(a)=-1$ and $\mathrm{Tr}_{q^n/q^{r_2}}(\alpha a)=1$, or such that $\mathrm{Tr}_{q^n/q^{r_2}}(a)=-1$ and $\mathrm{Tr}_{q^n/q^{r_1}}(a/\alpha)=1$,
then $L_{r_1}$ and $L_{r_2}$ share at least one point other than their head $H$.
\end{theorem}

\begin{proof}
Let $\beta=-1/\alpha$.
Suppose that there exists $a\in \F_{q^n}$ such that $\mathrm{Tr}_{q^n/q^{r_1}}(a)=-1$ and $\mathrm{Tr}_{q^n/q^{r_2}}(\alpha a)=1$.
By Remark \ref{rk:Hilbert}, $L_{r_1}\cap L_{r_2}\ne\{H\}$ if and only if there exist $\bar{x},\bar{y}\in\F_{q^n}$ such that $F(\bar{x},\bar{y})=0$, where
\[F(X,Y)= X^{q^{r_1}}-X+\beta(Y^{q^{r_2}}-Y)+\gamma_1+\beta\gamma_2, \]
for $\gamma_1=-a$ and any fixed $\gamma_2\in\F_{q^n}$ with $\mathrm{Tr}_{q^n/q^{r_2}}(\gamma_2)=1$.
As $X^{q^{r_1}}-X=\prod_{\varepsilon \in \F_{q^{r_1}}} (X-\varepsilon)$, if there exists $\bar{y}\in \F_{q^n}$ such that
\begin{equation}\label{eq:traceinty} \beta(\bar{y}^{q^{r_2}}-\bar{y})+\gamma_1+\beta \gamma_2=0,
\end{equation}
then $F(\bar{x},\bar{y})=0$ for every $\bar{x}\in \F_{q^{r_1}}$. Clearly, \eqref{eq:traceinty} admits a solution if and only if
\[ \mathrm{Tr}_{q^n/q^{r_2}}(\beta^{-1}\gamma_1+\gamma_2)=0, \]
i.e. $\mathrm{Tr}_{q^n/q^{r_2}}(\beta^{-1}\gamma_1)=-1$.
Suppose that there exists $a\in \F_{q^n}$ such that $\mathrm{Tr}_{q^n/q^{r_2}}(a)=-1$ and $\mathrm{Tr}_{q^n/q^{r_1}}(a/\alpha)=1$.
Arguing as above, if $\bar{y}\in\F_{q^{r_2}}$ and $\gamma_2=-a$ satisfies $\mathrm{Tr}_{q^n/q^{r_1}}(\beta\gamma_2)=-1$, then the claim follows.
\end{proof}

In the following result, we characterize the condition for $L_{r_1}$ and $L_{r_2}$ to share a further point other then their head when $\alpha \in \F_{q^{r_1}}\cdot \F_{q^{r_2}}=\{ a\cdot b \colon a \in \F_{q^{r_1}}, b \in \F_{q^{r_2}} \}$.

\begin{theorem}\label{th:secondo}
Let $\alpha=a\cdot b$, with $a \in \F_{q^{r_1}}$ and $b \in \F_{q^{r_2}}$, let $r_1,r_2\mid n$ with $(r_1,r_2)=1$.
The linear sets $L_{r_1}$ and $L_{r_2}$ share at least one point other than their head $H$ if and only if there exist
$\gamma_1,\gamma_2\in\F_{q^n}$ such that $\mathrm{Tr}_{q^n/q^{r_1}}(\gamma_1)=\mathrm{Tr}_{q^n/q^{r_2}}(\gamma_2)=1$ and $\mathrm{Tr}_{q^n/q}(a\gamma_1-\frac{\gamma_2}{b})=0$.
\end{theorem}

\begin{proof}
Assume $r_2 \leq r_1$.
As already noted in the proof of Theorem \ref{th:primo}, $L_{r_1}\cap L_{r_2}\ne\{H\}$ if and only if there exist $\bar{x},\bar{y}\in \F_{q^n}$ such that
\[ F(\bar{x},\bar{y})=\bar{x}^{q^{r_1}}-\bar{x}-\frac{1}{ab}(\bar{y}^{q^{r_2}}-\bar{y})+\gamma_1-\frac{\gamma_2}{ab}=0, \]
for some $\gamma_1,\gamma_2 \in \F_{q^n}$ with $\mathrm{Tr}_{q^{n}/q^{r_1}}(\gamma_1)=\mathrm{Tr}_{q^n/q^{r_2}}(\gamma_2)=1$.
Let $\varepsilon_1,\ldots,\varepsilon_q\in\mathbb{K}$ be the distinct roots of $T^q-T-(a\gamma_1-\frac{\gamma_2}{b})\in\F_{q^n}[T]$. Then
\[ \prod_{i=1}^q \left(\mathrm{Tr}_{q^{r_1}/q}(aX)-\mathrm{Tr}_{q^{r_2}/q}\left(\frac{Y}b \right)+\varepsilon_i\right)
= \left( \mathrm{Tr}_{q^{r_1}/q}(aX)-\mathrm{Tr}_{q^{r_2}/q}\left(\frac{Y}b \right)\right)^q +\]
\[- \left( \mathrm{Tr}_{q^{r_1}/q}(aX)-\mathrm{Tr}_{q^{r_2}/q}\left(\frac{Y}b \right)\right) -\prod_{i=1}^q \varepsilon_{i} = aF(X,Y). \]
Let $\mathcal{C}_i$ be the plane curve with affine equation $F_i(X,Y)=0$, where
\[
F_i(X,Y)=\mathrm{Tr}_{q^{r_1}/q}(aX)-\mathrm{Tr}_{q^{r_2}/q}\left(\frac{Y}b \right)+\varepsilon_i.
\]
The change of coordinates $(X,Y)\mapsto(aX,-Y/b)$ maps $\mathcal{C}_i$ to the curve $\mathcal{X}_i$ with equation $G_i(X,Y)=0$ where $G_i(X,Y)=F_i(X/a,-bY)$.

If $r_2=r_1=1$, then $G_i(X,Y)=X+Y+\varepsilon_i$. Hence, $\mathcal{X}_i$ has affine $\F_{q^n}$-rational points if and only if $\varepsilon_i\in\F_{q^n}$.

Suppose $r_2<r_1$ and write $r_1=\ell r_2+k$ with $0\leq k<r_2$.
Consider the rational map $\varphi_{r_2}:(X,Y)\mapsto(X+Y,X^{q^{r_2}})$. Then $\varphi_{r_2}$ maps $\mathcal{X}_i$ to the curve $\mathcal{X}_i^1$ with affine equation $G_i^1(X,Y)=0$, where
\[
G_i^1(X,Y)=\mathrm{Tr}_{q^{(\ell-1)r_2+k}/q}(X)+\mathrm{Tr}_{q^{r_2}/q}(Y)+\varepsilon_i.
\]
Note that $\varphi_{r_2}$ induces a bijection between the affine $\F_{q^n}$-rational points of $\mathcal{X}_i$ and the affine $\F_{q^n}$-rational points of $\mathcal{X}_i^1$.
Let $\mathcal{X}_i^0=\mathcal{X}_i$. For every $j=1,\ldots,\ell$, let $\mathcal{X}_i^j$ be the curve $\varphi_{r_2}(\mathcal{X}_i^{j-1})$ having equation $G_i^j(X,Y)=0$, where
\[
G_i^j(X,Y)=\mathrm{Tr}_{q^{(\ell-j)r_2+k}/q}(X)+\mathrm{Tr}_{q^{r_2}/q}(Y)+\varepsilon_i,
\]
so that
$\mathcal{X}_i^\ell$ has equation $\mathrm{Tr}_{q^k/q}(X)+\mathrm{Tr}_{q^{r_2}/q}(Y)+\varepsilon_i=0$.
If $k=1$, then
\[ \mathcal{X}_i^{\ell}\colon\quad X=-\mathrm{Tr}_{q^{r_2}/q}(Y)-\varepsilon_i.
\]
If $k>1$, then perform the change of coordinates $\psi:(X,Y)\mapsto(Y,X)$, put $\mathcal{Y}_i^0=\psi(\mathcal{X}_i^\ell)$, and define $\mathcal{Y}_i^j=\varphi_k(\mathcal{Y}_i^{j-1})$ for every $j=1,\ldots,h$, where $h=\lfloor r_2/k\rfloor$.
The iterated application of this argument provides the curve $\mathcal{Z}_i$ with affine equation
\[
\mathcal{Z}_i\colon\quad X=-\mathrm{Tr}_{q^d/q}(Y)-\varepsilon_i,
\]
for some $d\geq1$; the affine $\F_{q^n}$-rational points of $\mathcal{C}_i$ are in one-to-one correspondence with the affine $\F_{q^n}$-rational points of $\mathcal{Z}_i$.

Therefore, $\mathcal{C}_i$ has an affine $\F_{q^n}$-rational points if and only if $\varepsilon_i\in\F_{q^n}$. Also, $\varepsilon_i\in\F_{q^n}$ for some $i\in\{1,\ldots,q\}$ if and only if $\varepsilon_i\in\F_{q^n}$ for all $i\in\{1,\ldots,q\}$.

Thus, $L_{r_1}\cap L_{r_2}\ne \{H\}$ is equivalent to the existence of some $\gamma_1,\gamma_2\in\F_{q^n}$ such that $\mathrm{Tr}_{q^n/q^{r_1}}(\gamma_1)=\mathrm{Tr}_{q^n/q^{r_2}}(\gamma_2)=1$ and $\mathrm{Tr}_{q^n/q}(a\gamma_1-\frac{\gamma_2}{b})=0$.



\end{proof}

Note that, when the maximum fields of linearity of $L_{r_1}$ and $\sigma(L_{r_2})$ coincide, the question whether $L_{r_1}\cap\sigma(L_{r_2})$ is non-empty has been analyzed in \cite[Section 6.1.1]{ShVVdV}: if $r_1=r_2$, then $L_{r_1}\cap \sigma(L_{r_2})\ne\emptyset$ if and only if $\alpha\in T=\{xy \colon x,y\in\mathbb{F}_{q^n},\mathrm{Tr}_{q^n/q^{r_1}}(x)=\mathrm{Tr}_{q^n/q^{r_1}}(y)=1\}$; see \cite[Proposition 6.3]{ShVVdV}.
Therefore, we suppose without restriction that $r_1\ne r_2$.

The linear sets $L_{r_1}$ and $\sigma(L_{r_2})$ share a point if and only if there exist $\bar{x}, \bar{y} \in \F_{q^n}^*$ such that $\Tr_{q^n/q^{r_1}}(\bar{x})\ne0$, $\Tr_{q^n/q^{r_2}}(\bar{y})\ne0$ and
\[ \frac{\bar{x}}{\Tr_{q^n/q^{r_1}}(\bar{x})} \frac{  \bar{y}}{\Tr_{q^n/q^{r_2}}(\bar{y})} = \alpha. \]
By Remark \ref{rk:Hilbert}, this is equivalent to the existence of an $\F_{q^n}$-rational affine point of the curve $\mathcal{C}_{r_1,r_2}$ with affine equation
\[ (U^{q^{r_1}}-U+\gamma_1) (V^{q^{r_2}}-V+\gamma_2)=\alpha, \]
for some $\gamma_1, \gamma_2 \in \F_{q^n}$ satisfying $\mathrm{Tr}_{q^n/q^{r_1}}(\gamma_1)=\mathrm{Tr}_{q^n/q^{r_2}}(\gamma_2)=1$.

\begin{theorem}\label{th:Cr1r2}
If $\frac{n}2\geq r_1+r_2+1$, then $L_{r_1}\cap\sigma(L_{r_2})\ne \emptyset$.
\end{theorem}

\begin{proof}
Consider the rational function field $\K(v)$, and define
\[ F(v) = \frac{\alpha}{v^{q^{r_2}}-v+\gamma_2}-\gamma_1. \]
The valuation of $F(v)$ is $-1$ at the $q^{r_2}$ distinct zeros of $v^{q^{r_2}}-v+\gamma_2$ in $\K(v)$, and $0$ at any other place of $\K(v)$.
Then, by Theorem \ref{th:artinschreier}, the equation $u^{q^{r_1}}-u=F(v)$ defines a generalized Artin-Schreier extension $\K(u,v)\colon\K(v)$ of degree $q^{r_1}$, with genus
\[ g(\K(u,v))=\frac{q^{r_1}-1}2(-2+2q^{r_2})=(q^{r_1}-1)(q^{r_2}-1). \]
Clearly $\K(u,v)$ is the function field of the curve $\cC_{r_1,r_2}$, and hence $\cC_{r_1,r_2}$ is absolutely irreducible with genus $g(\cC_{r_1,r_2})=(q^{r_1}-1)(q^{r_2}-1)$.

From Theorem \ref{th:hasseweil} and the numerical assumption follows
\[
|\C_{r_1,r_2}(\F_{q^n})| \geq q^n+1-2 g(\C_{r_1,r_2})\sqrt{q^n} >0.
\]
Then there exists an $\F_{q^n}$-rational place of $\cC_{r_1,r_2}$, centered at an $\F_{q^n}$-rational point $P$ of $\cC_{r_1,r_2}$.
The poles $Q$ of $u$ are not $\F_{q^n}$-rational. In fact, $Q$ is not a pole of $v$ and $\bar{v}=v(Q)$ satisfies $\bar{v}^{q^{r_2}}-\bar{v}+\gamma_2=0$; hence $\bar{v}\notin\F_{q^n}$, because $\mathrm{Tr}_{q^n/q^{r_2}}(\gamma_2)\ne0$.
Similarly, the poles of $v$ are not $\F_{q^n}$-rational.
Therefore, $P$ is an $\F_{q^n}$-rational affine point of $\cC_{r_1,r_2}$, whence $L_{r_1}\cap L_{r_2}\ne\emptyset$.
\end{proof}

\subsection{Binomial $\alpha X^{q^k}+\beta X$ case}\label{sec:binomial}

Let $g(X)=\alpha X^{q^k}+\beta X\in\tilde{\mathcal{L}}_{n,q}$ with $\alpha \neq 0$ and $k\geq1$, and $f(Y)=\sum_{i=0}^{d} a_i Y^{q^i}\in\tilde{\mathcal{L}}_{n,q}$, with $d\geq 0$ and $a_d\ne 0$.
We investigate the intersection between the $\F_q$-linear sets
\[  L_g=\{\langle (x,g(x)) \rangle_{\F_{q^n}} \colon x \in \F_{q^n}^* \},\quad L_f=\{\langle (y^{q^h},f(y)) \rangle_{\F_{q^n}} \colon y \in \F_{q^n}^* \}. \]

\begin{remark}
By Proposition {\rm \ref{prop:adjoint}}, $L_g=\{\langle(x,\hat{g}(x))\rangle_{\F_{q^n}}\colon x\in\F_{q^n}^*\}$, where $\hat{g}(x)=\alpha^{q^{n-k}}x^{q^{n-k}}+\beta x$ has $q$-degree $n-k$.
Thus, up to replacing $g(x)$ with $\hat{g}(x)$, we may always assume throughout this section that $k\leq n/2$.
\end{remark}

Let $\ell=\min\{i\colon a_i\ne 0\}$ and \[ \ell_2=\begin{cases}\min\{i>\ell\colon a_i\ne0\}&\textrm{if}\;\ell<d, \\ d&\textrm{if}\;\ell=d.\end{cases}\]
If $\ell\geq h$, define
\[
\bar{f}(y)=\sum_{i=\ell}^d a_i y^{q^{i-h}}.
\]
If $\ell\leq h$, define
\[
\tilde{f}(y)=\sum_{i=\ell}^d a_i y^{q^{i-\ell}}.
\]
Note that, if $\ell\geq h$, then
\[ L_f=\left\{\left\langle\left(y,\bar{f}(y))\right)\right\rangle_{\F_{q^n}}\colon z\in\F_{q^n}^*\right\}; \]
if $\ell\leq h$, then
\[ L_f=\left\{\left\langle\left(y,\tilde{f}(y))\right)\right\rangle_{\F_{q^n}}\colon z\in\F_{q^n}^*\right\}. \]

\begin{remark}\label{rem:aggiunto}
By Proposition {\rm \ref{prop:adjoint}},
\[
L_f=\left\{\left\langle\left(z,\sum_{i=\ell}^{d}a_i^{q^{n+h-i}}z^{q^{n+h-i}}\right)\right\rangle_{\F_{q^n}}\colon z\in\F_{q^n}^*\right\}\]
\begin{equation}\label{eq:vedimeglio}
=\left\{\left\langle\left(y^{q^{n-h}},\bar{\bar{f}}(y)\right)\right\rangle_{\F_{q^n}}\colon y\in\F_{q^n}^*\right\}\quad\textrm{where}\quad \bar{\bar{f}}(y)=\sum_{i=\ell}^{d}a_i^{q^{n+h-i}}y^{q^{n-i}}.
\end{equation}
Hence,
\begin{equation}\label{eq:vedimeglio2}
    \sigma(L_f)=\left\{\left\langle\left(\bar{\bar{f}}(y),y^{q^{n-h}}\right)\right\rangle_{\F_{q^n}}\colon y\in\F_{q^n}^*\right\}.
\end{equation}
In the hypothesis of the results of this section, the parameters $h,d,\ell$ play a crucial role since some bounds involving them are assumed.
When considering the shape \eqref{eq:vedimeglio} for $L_f$ or the shape \eqref{eq:vedimeglio2} for $\sigma(L_f)$, the role of $h$ in such bounds is played by $n-h$ or $0$ according to $h>0$ or $h=0$, respectively; the role of $d$ is played by $n-\ell$ or $n-\ell_2$ according to $\ell\ne0$ or $\ell=0$, respectively; the role of $\ell$ is played by $n-d$ or $0$ according to $\ell\ne0$ or $\ell=0$, respectively.
Therefore, the most convenient bound in the hypothesis of the results is taken into account, that is, a comparison can be made between the result obtained from the investigation of $f(y)$ and the corresponding corollary obtained from the investigation of $\bar{\bar{f}}(y)$.
\end{remark}

We first analyze in Proposition \ref{prop:h=dmon} and Theorem \ref{th:mon} the case when $f(y)$ is a monomial.
\begin{proposition}\label{prop:h=dmon}
Let $f(y)=a_d y^{q^d}$ and suppose that $h=d$.
Then $L_{g}\cap L_{f}\ne\emptyset$ if and only if $\N_{q^n/q^e}\left(\frac{a_d-\beta}{\alpha}\right)=1$ where $e=(n,k)$.
\end{proposition}

\begin{proof}
Note that $L_{f}=\{P\}$ with $P=\langle(1,a_d)\rangle_{\F_{q^n}}$, and $P\in L_g$ if and only if there exists $\bar{x}\in\F_{q^n}^*$ such that $\bar{x}^{q^k-1}=\frac{a_d-\beta}{\alpha}$. The claim follows.
\end{proof}


\begin{theorem}\label{th:mon}
Let $f(x)=a_d y^{q^d}$ and suppose that $h\ne d$.
\begin{itemize}
    \item If $\beta=0$, then $L_g\cap L_f\ne\emptyset$ if and only if $\mathrm{N}_{q^n/q^e}(a_d/\alpha)=1$, where $e=(n,k,d-h)$.
    \item If $\beta\ne0$ and $k+|d-h|\leq\frac{n}{2}$, then $L_g\cap L_f\ne\emptyset$.
\end{itemize}
\end{theorem}

\begin{proof}
Suppose $\beta=0$.
Then $L_g\cap L_f\ne\emptyset$ if and only if there exist $\bar{x},\bar{y}\in\F_{q^n}^*$ such that $\frac{\alpha}{a_d} \frac{\bar{x}^{q^k-1}}{\bar{y}^{q^h(q^{d-h}-1)}}= 1$ or $\frac{\alpha}{a_d}(\bar{x}^{q^k-1})(\bar{y}^{q^d(q^{h-d}-1)})= 1$ according to $d>h$ or $h>d$, respectively.
If such $\bar{x},\bar{y}$ exist, then clearly ${\rm N}_{q^n/q^e}(\alpha/a_d)=1$.
Conversely, suppose ${\rm N}_{q^n/q^e}(\alpha/a_d)=1$, so that $\alpha/a_d=\xi^{q^e-1}$ for some $\xi\in\fqn^*$.
Define $\gamma=\frac{q^d-1}{q^e-1}$ and $\epsilon=\frac{q^k-1}{q^e-1}$.
Since $(\gamma,\epsilon)$ is coprime with $q^n-1$, there exist $u,v\in\mathbb{Z}$ such that $\gamma v-\epsilon u\equiv 1\pmod{q^n-1}$.
Choose $\bar{x}=\xi^u$ and $\bar{y}^{q^h}=\xi^v$ if $d>h$, or $\bar{y}^{q^d}=1/\xi^v$ if $h>d$.
Then $\langle(\bar{x},\bar{y})\rangle_{\fqn}\in L_f\cap L_g$.
The first point follows.

Suppose $\beta\ne0$ and $d>h$.
Then $\bar{f}(y)=a_d y^{q^{d-h}}$ and $\mathcal{C}_{g,\bar{f}}$ has equation
\begin{equation}\label{eq:curveCgfmon} X^{q^k-1}=\alpha^{-1}\left(a_d Y^{q^{d-h}-1}-\beta\right). \end{equation}
Let $F(Y)=\alpha^{-1}\left(a_d Y^{q^{d-h}-1}-\beta\right)$, and consider the rational functional field $\mathbb{K}(y)$. Then the rational function $F(y)$ has valuation $-(q^{d-h}-1)$ at the pole of $y$, and valuation $1$ at the $q^{d-h}-1$ zeros of $y-\lambda$ for $\lambda\in\mathbb{K}$ satisfying $F(\lambda)=0$. Every other place of $\mathbb{K}(y)$ is neither a zero nor a pole of $F(y)$.
By Theorem \ref{th:kummer}, the equation $x^{q^k-1}=F(y)$ defines a Kummer extension $\mathbb{K}(x,y)\colon\mathbb{K}(y)$ of degree $q^k-1$, whose genus is equal to
\[ 1+(q^k-1)(0-1)+\frac{1}{2}\left((q^{d-h}-1)(q^k-2) + q^k-1-(q^{(k,d-h)}-1)\right) \]
\begin{equation}\label{eq:genere} = \frac{(q^k-2)(q^{d-h}-3)+q^k-q^{(k,d-h)}}{2}. \end{equation}
Therefore, $\mathcal{C}_{g,\bar f}$ is absolutely irreducible with genus $g(\mathcal{C}_{g,\bar f})$ as in \eqref{eq:genere}.
The number of poles of $x$ or $y$ on $\mathcal{C}_{g,\bar f}$ is $q^{(k,d-h)}-1$, the number of zeros of $y$ is $q^k-1$, and the number of zeros of $x$ is $q^{d-h}-1$.
By Theorem \ref{th:hasseweil} and the numerical assumption, we have
\[ |\mathcal{C}_{g,\bar f}(\F_{q^n})|\geq q^n+1-2g(\mathcal{C}_{g,\bar f})\sqrt{q^n}>q^{(k,d-h)}-1 + q^k-1 + q^{d-h}-1. \]
Thus, there exists an $\F_{q^n}$-rational place $P$ of $\mathcal{C}_{g,\bar{f}}$ which is neither a zero nor pole of $x$ or $y$.
Therefore, the center of $P$ is an $\F_{q^n}$-rational affine point $(\bar{x},\bar{y})$ of $\mathcal{C}_{g,\bar f}$ with $\bar{x}\ne0$ and $\bar{y}\ne0$; the claim follows.

Suppose $\beta\ne0$ and $d<h$.
Then $L_f=\{\langle (a_d^{-1}y^{q^{h-d}},y)\rangle_{\F_{q^n}} \colon y \in \F_{q^n}^*\}$.
Therefore, if $f'(y)=a_d^{-1}y^{q^{h-d}}$, the curve $\mathcal{X}_{g,f'}$ is given by
\begin{equation}\label{eq:curve1a} X^{q^k-1}=\alpha^{-1} (a_d Y^{1-q^{h-d}}-\beta). \end{equation}
Clearly, the projectivity $\psi:(X,Y)\mapsto(X,1/Y)$ maps $\mathcal{X}_{g,f^\prime}$ to the curve $\mathcal{C}_{g,\bar f}$ in \eqref{eq:curveCgfmon} (where $d-h$ is replaced by $h-d$).
Therefore, arguing as above, the claim follows.
\end{proof}

By the previous results, we can assume from now on that $f(y)$ is not a monomial.
Next proposition considers the case when $f(y)$ is a binomial of a particular shape.

\begin{proposition}\label{prop:sabbassa}
Let $f(y)=a_d y^{q^d}+a_\ell y^{q^\ell}$ where $d\ne\ell$ and one of the following holds: either $d=h$ and $a_d=\beta$, or $\ell=h$ and $a_\ell=\beta$.
Let
\begin{equation}\label{eq:t} t=\begin{cases} \ell & \textrm{if}\quad h=d, \\ d & \textrm{if}\quad h=\ell.  \end{cases}\end{equation}
Then $L_f \cap L_g\neq \emptyset$ if and only if $\N_{q^n/q^e}(a_t/\alpha)= 1$, where $e=(n,k,t-h)$.
\end{proposition}

\begin{proof}
We have $L_f\cap L_g\ne\emptyset$ if and only if \[ \bar{x}^{q^k-1}= \frac{a_t \bar{y}^{q^t}}{\alpha\bar{y}^{q^h}}=\frac{a_t}{\alpha} \bar{y}^{q^h(q^{t-h}-1)} \]
for some $\bar{x},\bar{y}\in\F_{q^n}^*$.
The claim follows.
\end{proof}

Next results deal with $L_f\cap L_g$ for the remaining shapes of $f(y)$, considering separately the cases $h\leq\ell$ and $h>\ell$.

\begin{theorem}\label{th:h<ell_dritto}
Assume that $h\leq\ell$ and $f(y)$ is not a monomial.
If $f(y)=a_d y^{q^d}+a_\ell y^{q^\ell}$ and $h=\ell$, assume also that $a_h\ne\beta$.
Let
\[
m_h=\begin{cases}
0,&\textrm{if}\quad a_h\ne\beta, \\
\ell-h,&\textrm{if}\quad a_h=\beta=0, \\
\ell_2-h,&\textrm{if}\quad a_h=\beta\ne0. \\
\end{cases}
\]
If
\[
\max\left\{ k+d-h-m_h,\frac{d-h}{2}\right\}\leq\begin{cases} \frac{n}{2} & \textrm{if}\quad m_h\leq\frac{d-h}{2}, \\ \frac{n}{2}-1 & \textrm{if}\quad m_h>\frac{d-h}{2}, \end{cases}
\]
then $L_{g}$ and $L_f$ share at least one point.
\end{theorem}

\begin{proof}
The curve $\mathcal{C}_{g,\bar{f}}$ has equation $X^{q^k-1}=F(Y)$, where
\[ F(Y)=\frac{\bar{f}(Y)-\beta Y}{\alpha Y}.\]
The valuations of the rational function $F(y)$ at the places of the rational function field $\mathbb{K}(y)$ are as follows.
The number of nonzero roots $\eta$ of $F(Y)$ is $q^{d-h-m_h}-1$; the valuation of $F(y)$ at the zero of $y-\eta$ is $q^{m_h}$.
The valuation of $F(y)$ is $-(q^{d-h}-1)$ at the pole of $y$, and $q^{m_h}-1$ at the zero of $y$.
The function $F(y)$ does not have any other zero or pole.
Thus, by Theorem \ref{th:kummer}, the equation $x^{q^k-1}=F(y)$ defines a Kummer extension $\mathbb{K}(x,y)\colon\mathbb{K}(y)$ of degree $q^k-1$, which is clearly the function field of $\mathcal{C}_{g,\bar{f}}$. Then $\mathcal{C}_{g,\bar{f}}$ is absolutely irreducible, with genus
\[g(\mathcal{C}_{g,\bar{f}})=1+(q^k-1)(0-1)+\frac{1}{2}\big((q^{d-h-m_h}-1)(q^k-2)+q^k-1-(q^{(k,d-h)}-1)\]
\[+q^k-1-(q^{(k,m_h)}-1)\big) = \frac{(q^k-2)(q^{d-h-m_h}-3)+2q^k-q^{(k,d-h)}-q^{(k,m_h)}}{2}. \]
On the curve $\mathcal{C}_{g,\bar{f}}$ there are exactly $q^{(k,d-h)}-1$ poles of $x$ or $y$, $q^{(k,m_h)}-1$ zeros of $y$, and $q^{d-h}-1$ zeros of $x$.
By Theorem \ref{th:hasseweil} and the numerical assumption, we have
\[ |\mathcal{C}_{g,\bar{f}}(\mathbb{F}_{q^n})|\geq q^n+1-2g(\mathcal{C}_{g,\bar{f}})\sqrt{q^n}> q^{(k,d-h)}-1 + q^{(k,m_h)}-1 + q^{d-h}-1. \]
Therefore, there exists an $\F_{q^n}$-rational affine point $(\bar{x},\bar{y})$ of $\mathcal{C}_{g,\bar{f}}$ with $\bar{x}\ne0$ and $\bar{y}\ne0$, whence the claim is proved.
\end{proof}



If $f(y)$ is not a monomial, define $\ell_3=\max\{i<d\colon a_i\ne0\}$.

\begin{theorem}\label{th:h>ell}
Assume that $h>\ell$ and $f(y)$ is not a monomial.
If $f(y)=a_d y^{q^d}+a_\ell y^{q^\ell}$ and $h=d$, assume also that $a_h\ne\beta$.
Let
\[
 m_h=\begin{cases} \max\{d,h\} & \textrm{if}\quad a_d\ne\beta\;\textrm{ or }\;d\ne h, \\ \ell_3 & \textrm{if}\quad a_d=\beta\;\textrm{ and }\; d=h. \end{cases}
\]
If $k+m_h-\ell\leq n/2$, then $L_g\cap L_f\ne\emptyset$.
\end{theorem}

\begin{proof}
The curve $\mathcal{C}^{h-\ell}_{g,\tilde{f}}$ has equation $X^{q^k-1}=F(Y)$ where \[ F(Y)=\frac{\tilde{f}(Y)-\beta Y^{q^{h-\ell}}}{\alpha Y^{q^{h-\ell}}}.\]
Then $\tilde{f}(Y)-\beta Y^{q^{h-\ell}}$ has exactly $q^{m_h-\ell}-1$ nonzero simple roots $\eta\in\mathbb{K}$; the valuation of $F(y)$ at the zero of $y-\eta$ in $\mathbb{K}(y)$ is $1$.
Also, the valuation of $F(y)$ is $-q^{m_h-\ell}+q^{h-\ell}=-q^{h-\ell}(q^{m_h-h}-1)$ at the pole of $y$, and $-(q^{h-\ell}-1)$ at the zero of $y$.
The function $F(y)$ has no other zeros or poles in $\mathbb{K}(y)$.
Then, by Theorem \ref{th:kummer}, the equation $x^{q^k-1}=F(y)$ defines a Kummer extension $\mathbb{K}(x,y)\colon\mathbb{K}(y)$ of degree $q^k-1$, which is clearly the function field of $\mathcal{C}^{h-\ell}_{g,\tilde{f}}$.
Thus, $\mathcal{C}^{h-\ell}_{g,\tilde{f}}$ is absolutely irreducible; by Theorem \ref{th:kummer},
\[ g(\mathcal{C}^{h-\ell}_{g,\tilde{f}})=\frac{(q^k-2)(q^{m_h-\ell}-3)+2q^k-q^{(k,h-\ell)}-q^{(k,m_h-h)}}{2}. \]
The function $y$ has exactly $q^{(k,m_h-h)}-1$ poles and $q^{(k,h-\ell)}-1$ zeros on $\mathcal{C}_{g,\tilde{f}}$; the poles of $x$ are zeros or poles of $y$; the number of zeros of $x$ which are not poles of $y$ is $q^{m_h-\ell}-1$.
By Theorem \ref{th:hasseweil} and the numerical assumption, we have
\[ |\mathcal{C}^{h-\ell}_{g,\tilde{f}}|\geq q^n+1-2g(\mathcal{C}^{h-\ell}_{g,\bar{f}})\sqrt{q^n}> q^{(k,m_h-h)}-1 + q^{(k,h-\ell)}-1 + q^{m_h-\ell}-1. \]
Hence,
there exists an $\F_{q^n}$-rational affine point $(\bar{x},\bar{y})$ of $\mathcal{C}^{h-\ell}_{g,\tilde{f}}$ with $\bar{x}\ne0$ and $\bar{y}\ne0$. The claim is now proved.
\end{proof}

Now we investigate the intersection between the $\fq$-linear sets
\[
L_g=\{\langle(x,g(x))\rangle_{\F_{q^n}}\colon x\in\F_{q^n}^*\},\quad \sigma(L_f)=\{\langle(f(y),y)\rangle_{\F_{q^n}}\colon y\in\F_{q^n}^*\}.
\]

Proposition \ref{prop:pointinv} considers the case when $f(y)$ is a monomial.

\begin{proposition}\label{prop:pointinv}
Let $f(y)=a_d y^{q^d}$.
\begin{itemize}
    \item If $d=h$ or $\beta=0$, then $L_g\cap \sigma(L_f)\ne\emptyset$ if and only if $\mathrm{N}_{q^n/q^e}\left(\frac{1-\beta a_d}{\alpha a_d}\right)=1$, where $e=(n,k,d-h)$.
    \item If $d\ne h$, $\beta\ne0$ and $k+|d-h|\leq\frac{n}{2}$, then $L_g\cap\sigma(L_f)\ne\emptyset$.
\end{itemize}
\end{proposition}

\begin{proof}
As
\[
\sigma(L_f)=\{\langle(y^{q^d},a_d^{-1}y^{q^h})\rangle_{\F_{q^n}}\colon y\in\F_{q^n}^*\},
\]
the claim follows from Proposition \ref{prop:h=dmon} and Theorem \ref{th:mon}.
\end{proof}


We now consider the case when $f(y)$ is not a monomial and $h\leq\ell$.

\begin{theorem}\label{th:h<ell_girato}
Assume that $h\leq \ell$ and $f(y)$ is not a monomial.
Let
\[
m_h=\begin{cases}
\ell_2 & \textrm{if}\quad \beta a_h=1, \\ h & \textrm{if}\quad \beta\ne0\quad \textrm{and}\quad\beta a_h\ne1, \\
d & \textrm{if}\quad \beta=0.
\end{cases}
\]
If $k+d-\min\{m_h,\ell\}+1\leq n/2$, then $L_g\cap\sigma(L_f)\ne\emptyset$.
\end{theorem}

\begin{proof}
The curve $\mathcal{X}_{g,\bar{f}}^0$ has plane model $X^{q^k-1}=F(Y)$, where
\[ F(Y)=\frac{Y-\beta \bar{f}(Y)}{\alpha\bar{f}(Y)}. \]
The valuation of the rational function $F(y)$ at the places of the rational function field $\mathbb{K}(y)$ is as follows. The valuation of $F(y)$ is $-(q^{\ell-h}-1)$ or $q^{\ell-h}(q^{\ell_2-\ell}-1)$ at the zero $P_0$ of $y$, according to $\beta a_h\ne1$ or $\beta a_h=1$, respectively.
The valuation of $F(y)$ at the pole $P_\infty$ of $y$ is either $0$ or $q^{d-h}-1$ according to $\beta\ne0$ or $\beta=0$, respectively.
The number of nonzero distinct roots $\eta\in\mathbb{K}$ of $Y-\beta\bar{f}(Y)$ is $0$, or $q^{d-h}-1$, or $q^{d-\ell_2}-1$, according to $\beta=0$, or $\beta\ne0$ and $\beta a_h\ne1$, or $\beta a_h=1$, respectively; the valuation of $F(y)$ at the zero $P_\eta$ of $y-\eta$ is $1$ or $q^{\ell_2-h}$ according to $\beta a_h\ne1$ or $\beta a_h=1$, respectively.
The valuation of $F(y)$ at the zero of $y-\xi$ is $-q^{\ell-h}$, when $\xi$ ranges over the $q^{d-\ell}-1$ nonzero distinct roots of $\bar{f}(Y)$.
Then, by Theorem \ref{th:kummer}, the equation $x^{q^k-1}=F(y)$ defines a Kummer extension $\mathbb{K}(x,y)\colon\mathbb{K}(y)$ of degree $q^k-1$, which is clearly the function field of $\mathcal{X}_{g,\bar{f}}^0$.
Thus, $\mathcal{X}_{g,\bar{f}}^0$ is absolutely irreducible; by Theorem \ref{th:kummer},
\[
g(\mathcal{X}_{g,\bar{f}}^0)= \frac{(q^k-2)(q^{d-m_h}+q^{d-\ell}-4)+\varepsilon_{h}}{2},
\]
where $\varepsilon_h=2q^k-2-(q^k-1,v_{P_\infty}(F(y)))-(q^k-1,v_{P_0}(F(y)))$.
The function $y$ has exactly $(q^k-1,v_{P_\infty}(F(y)))$ poles and $(q^k-1,v_{P_0}(F(y)))$ zeros on the curve $\mathcal{X}_{g,\tilde{f}}^0$. The function $x$ has exactly $q^{d-\ell}-1$ poles. The zeros of $x$ are the places over $P_\eta$ for some $\eta$, and, if $\beta=0$, also the places over $P_\infty$; hence, the zeros of $x$ which are not poles of $y$ are $q^{d-m_h}-1$. Altogether, denote by $Z$ the number of zeros and poles of $x$ and $y$ on $\mathcal{X}_{g,\bar{f}}^0$.

By Theorem \ref{th:hasseweil} and the numerical assumption,
\[ |\mathcal{X}_{g,\bar{f}}^0(\mathbb{F}_{q^n})|\geq q^n+1-2g(\mathcal{X}_{g,\bar{f}}^0)>Z. \]
Thus, there exists an $\mathbb{F}_{q^n}$-rational affine point of $\mathcal{X}_{g,\bar{f}}^0$ with nonzero coordinates, and the claim follows.
\end{proof}

\begin{theorem}\label{th:h>lrib}
Assume that $h>\ell$ and $f(y)$ is not a monomial. Let
\[
m_h=\begin{cases}
\ell, & \textrm{if}\quad \beta=0;\\
\max\{d,h\}, & \textrm{if}\quad \beta\ne0\quad \textrm{and either}\quad h\ne d\quad\textrm{or}\quad \beta a_h\ne1;\\
\ell_3,&\textrm{if}\quad h=d\quad\textrm{and}\quad \beta a_h=1.\\
\end{cases}
\]
If $k+\max\{m_h,d\}-\ell+1\leq n/2$, then $L_g\cap\sigma(L_f)\ne\emptyset$.
\end{theorem}

\begin{proof}
The curve $\mathcal{X}_{g,\tilde{f}}^{h-\ell}$ has equation $X^{q^k-1}=F(Y)$ where \[ F(Y)=\frac{Y^{q^{h-\ell}}-\beta\tilde{f}(Y)}{\alpha\tilde{f}(Y)}.\]
The rational function $F(y)$ has the following valuations at the places of the rational function field $\mathbb{K}(y)$.
The valuation of $F(y)$ at the zero $P_0$ of $y$ is either $0$ or $q^{h-\ell}-1$ according to $\beta\ne0$ or $\beta=0$, respectively.
The valuation of $F(y)$ at the pole $P_\infty$ of $y$ is
\[
v_{P_\infty}(F(y))=\begin{cases}
0,&\textrm{if}\quad \beta\ne0\quad\textrm{and}\quad h<d; \\
0, & \textrm{if}\quad\beta\ne0\quad\textrm{and}\quad h=d\quad\textrm{and}\quad\beta a_d\ne1; \\
q^{\ell_3 - \ell}(q^{d-\ell_3}-1),&\textrm{if}\quad h=d\quad\textrm{and}\quad \beta a_d=1;\\
-q^{d-\ell}(q^{h-d}-1),&\textrm{if}\quad\beta=0\quad\textrm{or}\quad h>d.\\
\end{cases}
\]
The number of nonzero roots $\eta\in\mathbb{K}$ of $Y^{q^{h-\ell} }-\beta\tilde{f}(Y)$ is: $0$, if $\beta=0$; $q^{\max\{d,h\}-\ell}-1$, if $\beta\ne0$ and either $h\ne d$ or $\beta a_d\ne 1$; $q^{\ell_3-\ell}-1$, if $h=d$ and $\beta a_d=1$. The valuation of $F(y)$ at the zero of $y-\eta$ is $1$.
The number of nonzero roots $\xi\in\mathbb{K}$ of $\tilde{f}(Y)$ is $q^{d-\ell}-1$, and the valuation of $F(y)$ at the zero of $y-\xi$ is $-1$.

Then, by Theorem \ref{th:kummer}, the equation $x^{q^k-1}=F(y)$ defines a Kummer extension $\mathbb{K}(x,y)\colon\mathbb{K}(y)$ of degree $q^k-1$, which is the function field of the curve $\mathcal{X}_{g,\tilde{f}}^{h-\ell}$. Thus, $\mathcal{X}_{g,\tilde{f}}^{h-\ell}$ is absolutely irreducible; by Theorem \ref{th:kummer}, its genus is
\[
g(\mathcal{X}_{g,\tilde{f}}^{h-\ell})=\frac{(q^k-2)(q^{m_h-\ell}+q^{d-\ell}-4)+\varepsilon_h}{2},
\]
where $\varepsilon_h=2q^k-2-(q^k-1,v_{P_{\infty}}(F(y)))-(q^k-1,v_{P_{0}}(F(y)))$.
The number $Z$ of zeros or poles on $\mathcal{X}_{g,\tilde{f}}^{h-\ell}$ of the coordinate functions is given as follows.
The number of zeros and poles of $y$ is respectively  $(q^k-1,v_{P_{0}}(F(y)))$ and $(q^k-1,v_{P_{\infty}}(F(y)))$; the number of poles of $x$ which are not poles of $y$ is $q^{d-\ell}-1$; the number of zeros of $x$ which are neither zeros nor poles of $y$ is $q^{m_h-\ell}-1$.
From Theorem \ref{th:hasseweil} and the numerical assumption follows
\[
|\mathcal{X}_{g,\tilde{f}}^{h-\ell}(\F_{q^n})|\geq q^n+1-2g(\mathcal{X}_{g,\tilde{f}}^{h-\ell})\sqrt{q^n}> Z.
\]
Thus, there exists an affine $\F_{q^n}$-rational point of $\mathcal{X}_{g,\tilde{f}}^{h-\ell}$ with nonzero coordinates. The claim follows.
\end{proof}

Corollary \ref{cor:h<ell_dritto_aggiunto} follows from Theorem \ref{th:h<ell_dritto} and Remark \ref{rem:aggiunto}.

\begin{corollary}\label{cor:h<ell_dritto_aggiunto}
Assume that $f(y)$ is not a monomial.
If $\ell\ne0$, assume $h\geq d$; if $\ell=0$, assume $h=0$.
If $f(y)=a_d y^{q^d}+a_\ell y^{q^\ell}$ and one of the following two cases holds:
\begin{itemize}
    \item $\ell\ne0$ and $h=d$,
    \item $\ell=0$ and $h=0$,
\end{itemize}
then assume also $a_h\ne\beta$.
Let
\[
\hat{m}_{h}=\begin{cases}
0,&\textrm{if}\quad a_{h}\ne\beta, \\
h-d,&\textrm{if}\quad a_{h}=\beta=0, \\
h-\ell_3,&\textrm{if}\quad a_{h}=\beta\ne0\quad\textrm{and}\quad \ell\ne0, \\
n-d,&\textrm{if}\quad a_h=\beta\ne0\quad\textrm{and}\quad \ell=0.
\end{cases}
\]
If either $\ell\ne0$ and
\[
\max\left\{ k+h-\ell-\hat{m}_h ,\frac{h-\ell}{2}\right\}\leq\begin{cases} \frac{n}{2} & \textrm{if}\quad \hat{m}_h\leq\frac{h-\ell}{2}, \\ \frac{n}{2}-1 & \textrm{if}\quad \hat{m}_h>\frac{h-\ell}{2}, \end{cases}
\]
or $\ell=0$ and
\[
\max\left\{ k+n-\ell_2-\hat{m}_h ,\frac{n-\ell_2}{2}\right\}\leq\begin{cases} \frac{n}{2} & \textrm{if}\quad \hat{m}_h\leq\frac{n-\ell_2}{2}, \\ \frac{n}{2}-1 & \textrm{if}\quad \hat{m}_h>\frac{n-\ell_2}{2}, \end{cases}
\]

then $L_{g}$ and $L_f$ share at least one point.
\end{corollary}

From Theorem \ref{th:h>ell} and Remark \ref{rem:aggiunto}, we have the following result.

\begin{corollary}\label{cor:h>elldual}
Assume that $f(y)$ is not a monomial and $h>0$. If $\ell\ne0$, assume $h<d$.
If $f(y)=a_d y^{q^d}+a_\ell y^{q^\ell}$ and one of the following two cases holds:
\begin{itemize}
    \item $\ell\ne0$ and $h=\ell$,
    \item $\ell=0$ and $h=d$,
\end{itemize}
then assume also $a_h\ne\beta$.
Define
\[ \hat{m}_h=\begin{cases} n-\min\{h,\ell\}, & \textrm{if}\; \quad\ell\ne0,\quad\textrm{and}\quad \textrm{either}\quad a_{h}\ne\beta\quad\textrm{ or }\quad\ell\ne h; \\ n-\min\{h,\ell_2\}, &  \textrm{if}\quad \ell=0,\quad\textrm{and}\quad \textrm{either}\quad  a_{h}\ne\beta\quad\textrm{or}\quad\ell_2\ne h; \\
n-\ell_2,&\textrm{if}\quad \ell\ne0\quad \textrm{and}\quad a_h=\beta\quad \textrm{and}\quad h=\ell. \\
\end{cases}
\]
If either $\ell\ne 0$ and $k+\hat{m}_h+d\leq\frac{3n}{2}$, or $\ell=0$ and $k+\hat{m}_h\leq\frac{n}{2}$, then $L_g\cap L_f\ne\emptyset$.
\end{corollary}

Theorem \ref{th:h<ell_girato} and Remark \ref{rem:aggiunto} yield the following result.

\begin{corollary}\label{cor:rovultimo}
Assume that $f(y)$ is not a monomial.
If $\ell\ne0$, assume $h\geq d$; if $\ell=0$, assume $h=0$.
Let
\[
\hat{m}_h=\begin{cases}
0 & \textrm{if} \quad \beta= 0, \\
h-\ell & \textrm{if}\quad \ell\ne 0 \quad \textrm{and}\quad \beta a_h\ne 1, \\
n-\ell_2 & \textrm{if}\quad \ell= 0 \quad \textrm{and}\quad \beta a_h\ne 1, \\
\ell_3-\ell & \textrm{if}\quad \ell \ne 0 \quad\textrm{and}\quad\beta a_h=1, \\
d-\ell_2 & \textrm{if}\quad \ell = 0 \quad\textrm{and}\quad\beta a_h=1. \\
\end{cases}
\]
If either $\ell\ne0$ and $k+\max\{\hat{m}_h,d-\ell\}+1\leq n/2$, or $\ell=0$ and $k+\max\{\hat{m}_h,n-\ell_2\}+1\leq n/2$, then $L_g\cap \sigma(L_f)\neq \emptyset$.
\end{corollary}

Finally, the following result is a consequence of Remark \ref{rem:aggiunto} applied to Theorem \ref{th:h>lrib}.

\begin{corollary}\label{cor:ultimissimo}
Assume that $f(y)$ is not a monomial and $h>0$. If $\ell\ne0$, assume $h<d$.
Let
\[
t=\begin{cases}
n,&\textrm{if}\quad f(y)\quad\textrm{is a binomial}, \\
\min\{i>\ell_2\colon a_i\ne0\},&\textrm{otherwise};
\end{cases}\]
\[
\hat{m}_h=\begin{cases}
d,&\textrm{if}\quad\ell\ne0\quad\textrm{and}\quad\beta=0,\\
n,&\textrm{if}\quad\ell=0\quad\textrm{and}\quad\beta=0,\\
\min\{\ell,h\},&\textrm{if}\quad\ell\ne0\quad\textrm{and}\quad\beta\ne0\quad\textrm{and either}\quad h\ne\ell\quad\textrm{or}\quad\beta a_h\ne1,\\
\min\{\ell_2,h\},&\textrm{if}\quad\ell=0\quad\textrm{and}\quad\beta\ne0\quad\textrm{and either}\quad h\ne\ell_2\quad\textrm{or}\quad\beta a_h\ne1,\\
\ell_2,&\textrm{if}\quad\ell\ne0\quad\textrm{and}\quad h=\ell\quad\textrm{and}\quad \beta a_h=1,\\
t,&\textrm{if}\quad\ell=0\quad\textrm{and}\quad h=\ell_2\quad\textrm{and}\quad \beta a_h=1.\\
\end{cases}
\]
If either $\ell\ne0$ and $k+d-\min\{\hat{m}_h,\ell\}+1\leq n/2$, or $\ell=0$ and $k+n-\min\{\hat{m}_h,\ell_2\}+1\leq n/2$, then $L_g\cap\sigma(L_f)\ne\emptyset$.
\end{corollary}

\section{Asymptotic results on semifields}\label{sec:semifields}

A \emph{finite semifield} $(\mathcal{S},+,\circ)$ is a finite division algebra except that associativity of multiplication is not assumed. More precisely,
\begin{itemize}
    \item [(S1)] $(\mathcal{S},+)$ is a group with identity element $0$;
    \item [(S2)] $x\circ (y+z)= x\circ y + x \circ z$ and $(x+y)\circ z=x \circ z+ y \circ z$, for all $x,y,z \in \mathcal{S}$;
    \item [(S3)] $x\circ y=0$ implies $x=0$ or $y=0$;
    \item [(S4)] there exists $1 \in \mathcal{S}$ such that $1\circ x= x \circ 1=x$, for all $x \in \mathcal{S}$.
\end{itemize}
If (S4) does not hold, the structure $(\mathcal{S},+,\circ)$ is known as \emph{presemifield}.
The first examples of proper finite semifields were presented by Dickson in 1906 in \cite{Dickson}.
An essential notion in semifield theory is that of isotopism.
Two semifields $(\mathcal{S},+,\circ)$ and $(\mathcal{S}',+,\circ')$ are called \emph{isotopic} if there exist three non-singular linear transformations $F,G$ and $H$ from $\mathcal{S}$ to $\mathcal{S}'$ such that
\[ F(x) \circ' G(y) = H(x \circ y), \]
for all $x,y \in \mathcal{S}$. The triple $(F,G,H)$ is called an \emph{isotopism}.
The set of semifields isotopic to a semifield $\mathcal{S}$ is called the \emph{isotopism class} of $\mathcal{S}$.
The interest for the theory of semifields has increased because they appear in several areas of mathematics, such as finite geometry, difference sets, coding theory, cryptography and group theory. We refer to \cite{Kantor,LavPol} for further details on finite semifields and their applications.

\smallskip

In \cite{BallEL} Ball, Ebert and Lavrauw present a geometric construction of a finite semifield by considering a certain configuration of two subspaces with respect to a Desarguesian spread in a finite-dimensional vector space over a finite field.
They prove that every finite semifield can be obtained in this way. This construction was further investigated by Lavrauw and Sheekey in \cite{LavSheekey}, where they also give the notion of \emph{BEL-rank} of a semifield.
In particular, we are interested in those semifields having BEL-rank two that can be defined as follows.
We say that a finite semifield $(\mathcal{S},+,\circ)$ has \emph{BEL-rank two} if there exist two linearized polynomials $L_1(x)$ and $L_2(x)$ in $\tilde{\mathcal{L}}_{n,q}$ such that
\[ x \circ y =L_1(x)L_2(y)-xy \]
defines a presemifield $(\mathcal{S}_{L_1,L_2},+,\circ)$ isotopic to $(\mathcal{S},+,\circ)$.

\smallskip

Let $L_1(x)$ be a fixed linearized polynomial of $\tilde{\mathcal{L}}_{n,q}$.
Does there exist a linearized polynomial $L_2(x)$ in $\tilde{\mathcal{L}}_{n,q}$ such that $(\mathcal{S}_{L_1,L_2},+,\circ)$ is a semifield?
If $(\mathcal{S}_{L_1,L_2},+,\circ)$ is a semifield then
\[ L_1(x)L_2(y)-xy\neq 0 \]
for each $x,y \in \F_{q^n}^*$, or equivalently
the curve with equation
\[ \frac{L_2(Y)}{Y}=\frac{X}{L_1(X)} \]
does not have $\F_{q^n}$-rational affine points $(\overline{x},\overline{y})$ with $\overline{x},\overline{y}\neq 0$.
As noted in Section \ref{sec:intersections}, this is equivalent to the fact that the linear sets $\{\langle(L_1(x),x)\rangle_{\F_{q^n}}\colon x\in\F_{q^n}^*\}$ and $\{\langle (y,L_2(y)) \rangle_{\F_{q^n}} \colon y \in \F_{q^n}^* \}$ have a point in common.
Sheekey, Voloch and Van de Voorde in \cite[Corollary 6.5]{ShVVdV} give a condition on the degree of $L_2(x)$ when we choose $L_1(x)$ as the trace function,
ensuring that $(\mathcal{S}_{L_1,L_2},+,\circ)$ is not a semifield.
More precisely, they prove the following result.

\begin{theorem} {\rm \cite[Corollary 6.5]{ShVVdV}}\label{th:4.1}
If $(\mathcal{S}_{\mathrm{Tr}_{q^n/q},L_2},+,\circ)$ is a semifield, then $\deg L_2\geq q^{n/2-1}$.
\end{theorem}

As a direct consequence of Theorem \ref{th:Cr1r2} we obtain Corollary \ref{cor:semifield_tracce}.

\begin{corollary}\label{cor:semifield_tracce}
Let $r_1,r_2,n$ be three positive integers such that $r_1,r_2 \mid n$.
If $(\mathcal{S}_{\mathrm{Tr}_{q^n/q^{r_1}},\alpha \mathrm{Tr}_{q^n/q^{r_2}}},+,\circ)$ is a semifield, then $r_1+r_2>n/2-1$.
\end{corollary}

Let $r$ be a divisor of $n$ and suppose that $\mathcal{S}_{\mathrm{Tr}_{q^n/q},\mathrm{Tr}_{q^n/q^r}}$ is a semifield;
$\mathcal{S}$ is as in the assumption of Theorem \ref{th:4.1} when $L_2(y)=\mathrm{Tr}_{q^n/q^r}(y)$, and as in the assumption of Corollary when $r_1=1$, $r_2=r$, and $\alpha=1$.
Then the claim of Theorem \ref{th:4.1} provides $r\leq n/2+1$, whereas the claim of Corollary \ref{cor:semifield_tracce} provides $r>n/2-2$.
Therefore, if $\mathcal{S}_{\mathrm{Tr}_{q^n/q},\mathrm{Tr}_{q^n/q^r}}$ is a semifield, then either $n\geq8$ is even and $r=n/2$,
or $(n,r)$ is one of the pairs $(2,2)$, $(4,1)$, $(4,2)$, $(6,2)$, $(6,3)$, $(3,1)$, $(5,1)$, $(9,3)$.
We do not know whether these cases define a semifield or not.



\medskip

Let $f(x),g(y)\in\tilde{\mathcal{L}}_{n,q}$, where $f(x)=\sum_{i=0}^d a_i x^{q^i}$ with $a_d\ne0$, $g(y)=\alpha y^{q^k}+\beta y$ with $\alpha\ne0$, $k\geq1$.
As in Section \ref{sec:binomial}, define $\ell=\min\{i\colon a_i\ne0\}$.
If $f(x)$ is not a monomial, define $\ell_2=\{i>\ell\colon a_i\ne0\}$ and $\ell_3=\max\{i<d\colon a_i\ne0\}$;
as in Theorem \ref{th:h<ell_girato} and Corollary \ref{cor:rovultimo}, let
\[
m_0=\begin{cases}
\ell_2, & \textrm{if}\quad \beta a_0=1, \\ h, & \textrm{if}\quad \beta\ne0,\;\beta a_0\ne1, \\
d, & \textrm{if}\quad \beta=0;
\end{cases}
\quad\hat{m}_0=\begin{cases}
0, & \textrm{if} \quad \beta= 0, \\
n-\ell_2, & \textrm{if}\quad \beta a_0\ne 1, \\
d-\ell_2, & \textrm{if}\quad \beta a_0=1. \\
\end{cases}
\]

\begin{corollary}
Suppose that $(\mathcal{S}_{f,g},+,\circ)$ is a semifield.
\begin{itemize}
    \item If $f(x)$ is a monomial, and either $d=0$ or $\beta=0$, then $\N_{q^n/q^e}\left(\frac{1-\beta a_d}{\alpha a_d}\right)\ne1$, where $e=(n,k,d-h)$.
    \item If $f(x)$ is a monomial, $d\ne0$ and $\beta\ne0$, then $k+d>n/2$.
    \item If $f(x)$ is not a monomial, then $k+d-\min\{m_0,\ell\}+1>n/2$.
    \item If $f(x)$ is not a monomial and $\ell=0$, then $k+\max\{\hat{m}_0,n-\ell_2\}+1>n/2$.
\end{itemize}
\end{corollary}

\begin{proof}
The first and the second point of the claim follow from Proposition \ref{prop:pointinv}, the third one follows from Theorem \ref{th:h<ell_girato}, and the fourth one follows from Corollary \ref{cor:rovultimo}.
\end{proof}

Giovanni Zini and Ferdinando Zullo\\
Dipartimento di Matematica e Fisica,\\
Universit\`a degli Studi della Campania ``Luigi Vanvitelli'',\\
I--\,81100 Caserta, Italy\\
{{\em \{giovanni.zini,ferdinando.zullo\}@unicampania.it}}
\end{document}